\documentclass[12pt]{amsart}  

\usepackage{custom_preamble} 

\begin{document}

\title{Schur's colouring theorem for non-commuting pairs}

\author{\tsname}
\address{\tsaddress}
\email{\tsemail}

\begin{abstract}
For $G$ a finite non-Abelian group we write $c(G)$ for the probability that two randomly chosen elements commute and $k(G)$ for the largest integer such that any $k(G)$-colouring of $G$ is guaranteed to contain a monochromatic quadruple $(x,y,xy,yx)$ with $xy\neq yx$.  We show that $c(G) \rightarrow 0$ if and only if $k(G) \rightarrow \infty$.
\end{abstract}

\maketitle

\section{Introduction}

Our starting point is Schur's theorem \cite[Hilffssatz]{sch::4}, the proof of which adapts to give the following.
\begin{theorem}\label{thm.st}
Suppose that $G$ is a finite group and $\mathcal{C}$ is a cover of $G$ of size $k$.  Then there is a set $A \in \mathcal{C}$ with at least $c_k|G|^2$ triples $(x,y,xy) \in A^3$ where $c_k$ is a constant depending only on $k$.
\end{theorem}
The proof is a routine adaptation, but we shall not give it as the result as stated also follows from our next theorem.

If $G$ is non-Abelian then we might like to ask for quadruples $(x,y,xy,yx) \in A^4$ instead of triples and establishing the following such result (which we do in \S\ref{sec.trem}) is the main purpose of the paper.
\begin{theorem}\label{thm.main}
Suppose that $G$ is a finite group and $\mathcal{C}$ is a cover of $G$ of size $k$.  Then there is a set $A \in \mathcal{C}$ with $c_k|G|^2$ quadruples $(x,y,xy,yx) \in A^4$ where $c_k$ is a constant depending only on $k$.
\end{theorem}
When $G$ is non-Abelian we should like to ensure that at least one of the quadruples found in Theorem \ref{thm.main} has $xy\neq yx$, and to this end we define the \textbf{commuting probability} of a finite group $G$ to be
\begin{equation}\label{eqn.com}
c(G) :=\frac{1}{|G|^2}\sum_{x,y \in G}{1_{[xy=yx]}};
\end{equation}
in words it is the probability that a pair $(x,y) \in G^2$ chosen uniformly at random has $xy=yx$.  There are many nice results about the commuting probability -- see the introduction to \cite{heg::2} for details -- and it is an instructive exercise (see \cite{gus::0}) to check that if $c(G)<1$ then $c(G)\leq \frac{5}{8}$, so that if a group is non-Abelian there are `many' pairs that do not commute.  Despite this we have the following which we prove in \S\ref{sec.cy}.
\begin{proposition}\label{prop.col}
Suppose that $G$ is a finite group and $c(G) \geq \epsilon$.  Then there is a cover $\mathcal{C}$ of $G$ of size $\exp((2+o_{\epsilon\rightarrow 0}(1))\epsilon^{-1}\log \epsilon^{-1})$ such that if $A \in \mathcal{C}$ and $(x,y,xy,yx) \in A^4$ then $xy=yx$.
\end{proposition}
If $G$ is non-Abelian we write $k(G)$ for the \textbf{non-commuting Schur number} of $G$, that is the largest natural number such that for any cover $\mathcal{C}$ of $G$ of size $k(G)$ there is some $A \in \mathcal{C}$ and $(x,y,xy,yx) \in A^4$ with $xy\neq yx$.  (Note that since $G$ is assumed non-Abelian we certainly have $k(G) \geq 1$.)

The number $k(G)$ has been studied for a range of specific groups by McCutcheon in \cite{mcc::1} and we direct the interested reader there for examples and further questions.
\begin{theorem}
Suppose that $(G_n)_n$ is a sequence of non-Abelian groups.  Then $c(G_n) \rightarrow 0$ if and only if $k(G_n) \rightarrow \infty$.
\end{theorem}
\begin{proof}
The right to left implication follows immediately from Proposition \ref{prop.col}.  We can assume that $c_k$ is monotonically decreasing.  Suppose that $c(G_n) \rightarrow 0$ and there is a $k_0$ and an infinite set $S$ of naturals such that $k(G_n) < k_0$ for all $n \in S$.  Let $n \in S$ be such that $c(G_n)<c_{k_0}$ which can be done since $c(G_n) \rightarrow 0$ and $c_{k_0}>0$.

Since $k(G_n)<k_0$ there is a cover $\mathcal{C}$ of $G_n$ of size $k_0$ such that if $A\in \mathcal{C}$ and $(x,y,xy,yx) \in A^4$ then $xy=yx$.  By Theorem \ref{thm.main} there is $A \in \mathcal{C}$ such that $(x,y,xy,yx) \in A^4$ for at least $c_{k_0}|G_n|^2$ quadruples.  But then by design $xy=yx$ for all these pairs and so $c(G_n) \geq c_{k_0}$, a contradiction which proves the result.
\end{proof}

Before closing the section we need to acknowledge our debt to previous work.  In \cite{mcc::0} McCutcheon proves that $k(S_n) \rightarrow \infty$ as $n \rightarrow \infty$.  A short calculation shows that $c(S_n)\rightarrow 0$ as $n \rightarrow \infty$, and the possibility of showing that $k(G_n) \rightarrow \infty$ as $c(G_n) \rightarrow 0$ is identified by Bergelson and Tao in the remarks after \cite[Theorem 11]{bertao::0}.  Earlier, in \cite[Footnote 4]{bertao::0}, they also observe the significance of Neumann's work \cite{neu::1} which is the main idea behind the proof of Proposition \ref{prop.col}.

Write $D(G)$ for the smallest dimension of a non-trivial unitary representation of $G$.\footnote{This is called the quasirandomness of $G$ in \cite[Definition 1]{bertao::0} following the work of Gowers \cite{gow::2}}  In \cite[Corollary 8]{bertao::0} the authors show that $k(G_n) \rightarrow \infty$ as $D(G_n) \rightarrow \infty$, and in fact go further proving a density result.  For general finite groups there can be no density result; we refer the reader to the discussion after \cite[Theorem 11]{bertao::0} for more details.

\section{Proof of Theorem \ref{thm.main}}\label{sec.trem}

The proof of Theorem \ref{thm.main} is inspired by an attempt to translate the proof of \cite[Theorem 3.4]{bermcc::0} into a combinatorial setting.  There the authors use a recurrence theorem \cite[Theorem 5.2]{bermcczha::}; in its place we use a version of the Ajtai-Szemer{\'e}di Corners Theorem \cite{ajtsze::} for finite groups.  This was proved by Solymosi \cite[Theorem 2.1]{sol::0} using the triangle removal lemma.
\begin{theorem}\label{thm.triangle}
There is a function $f_\Delta:(0,1] \rightarrow (0,1]$ such that if $G$ is a finite group and $\mathcal{A} \subset G^2$ has size at least $\alpha |G|^2$ then
\begin{equation*}
S(\mathcal{A}):=\frac{1}{|G|^3}\sum_{x,y,z \in G}{1_{\mathcal{A}}(x,y)1_{\mathcal{A}}(zx,y)1_{\mathcal{A}}(x,yz)} \geq f_\Delta(\alpha).
\end{equation*}
\end{theorem}
\begin{proof}
Following the proof of \cite[Theorem 2.1]{sol::0}, form a tripartite graph with three copies of $G$ as the vertex sets (call them $V_1,V_2,V_3$) and joining $(x,y) \in V_1 \times V_2$ if and only if $(x,y) \in \mathcal{A}$; $(y,w) \in V_2 \times V_3$ if and only if $(y^{-1}w,y) \in \mathcal{A}$; and $(x,w) \in V_1\times V_3$ if and only if $(x,wx^{-1}) \in \mathcal{A}$.  The map $G^3 \rightarrow G^3;(x,y,w) \mapsto (x,y,y^{-1}wx^{-1})$ is a bijection and $(x,y,w)$ is a triangle in this graph if and only if $(x,y),(zx,y),(x,yz) \in \mathcal{A}$ where $z=y^{-1}wx^{-1}$.

It follows from \cite[Theorem 1.1]{tao::3} that one can remove at most $3\cdot o_{S(\mathcal{A}) \rightarrow 0}(|G|^2)=o_{S(\mathcal{A}) \rightarrow 0}(|G|^2)$ elements from $\mathcal{A}$ to make the graph triangle-free.  On the other hand if $(x,y) \in \mathcal{A}$ then $(x,y,xy)$ is a triangle in the above graph, hence we must have removed all elements from $\mathcal{A}$ and $\alpha |G|^2 \leq o_{S(\mathcal{A}) \rightarrow 0}(|G|^2)$ from which the result follows.
\end{proof}
There are a number of subtleties around the extent to which one can replace, say, $(zx,y)$ with $(xz,y)$, and we refer the reader to the papers of Solymosi \cite{sol::0} and Austin \cite{aus::1} for some discussion.

We take the convention, as we can, that the function $f_\Delta$ is monotonically increasing and $f_\Delta(x)\leq x$ for all $x \in (0,1]$.  Even with Fox's work \cite{fox::} we only have $f_\Delta(\alpha)^{-1} \leq T(O(\log \alpha^{-1}))$ in general.  However, when $G$ is Abelian much better bounds are known as a result of the beautiful arguments of Shkredov \cite{shk::1,shk::0,shk::6}.  It seems likely these could be adapted to give a bound with a tower of bounded height if the Fourier analysis is adapted to the non-Abelian setting in the same way as it is for Roth's theorem in \cite{san::13}.  Doing so would give a quantitative version of \cite[Theorem 10]{bertao::0} (see \cite[Remark 44]{bertao::0}), but the improvement to Theorem \ref{thm.main} would only be to replace a wowzer-type function with a tower as we shall see shortly.

We shall prove the following proposition from which Theorem \ref{thm.main} follows immediately on inserting the bound for $f_\Delta$ given by Theorem \ref{thm.triangle}.
\begin{proposition}\label{prop.m}
Suppose $G$ is a finite group and $\mathcal{C}$ is a cover of $G$ of size $k$.  Then there is a set $A \in \mathcal{C}$ with $(g^{(k+1)}(1))^2|G|^2$ quadruples $(x,y,xy,yx) \in A^4$ where $g^{(k+1)}$ is the $(k+1)$-fold composition of $g$ with itself, and $g:(0,1] \rightarrow (0,1]; \alpha \mapsto (3k)^{-1}f_\Delta(\alpha^k)$.
\end{proposition}
\begin{proof}
Write $A_1,\dots,A_k$ for the sets in $\mathcal{C}$ ordered so that they have density $\alpha_1 \geq \dots\geq \alpha_k$ respectively; since $\mathcal{C}$ is a cover we have $\alpha_1 \geq \frac{1}{k}$.  Let $r \in \{1,\dots,k\}$ be minimal such that
\begin{equation}\label{eqn.recurrence}
\frac{1}{3}f_\Delta(\alpha_1\cdots \alpha_r) \geq \alpha_{r+1}+\dots+\alpha_k,
\end{equation}
which is possible since the sum on the right is empty and so $0$ when $r=k$.  From minimality and the order of the $\alpha_i$s we have
\begin{equation*}
\alpha_{i+1} > \frac{1}{3k}f_\Delta(\alpha_1\cdots \alpha_i) \text{ for all } 1 \leq i \leq r-1.
\end{equation*}
The function $f_\Delta$ is monotonically increasing and $f_\Delta(x) \leq x$ for all $x \in (0,1]$ so it follows from the above that $\alpha_r \geq g^{(r)}(1) \geq g^{(k)}(1)$.

Now, suppose that $s_1,\dots,s_r \in G$ and write
\begin{equation*}
\mathcal{A}_i:=\{(x,y) \in G^2: xs_iy \in A_i\} \text{ for } 1 \leq i \leq r.
\end{equation*}
Then
\begin{equation*}
\E_{s_i \in G}{1_{\mathcal{A}_i}(x,y)} = \alpha_i \text{ for all } x,y \in G \text{ and } 1 \leq i \leq r,
\end{equation*}
and so
\begin{equation*}
\E_{s \in G^r}{\left|\bigcap_{i=1}^r{\mathcal{A}_i}\right|} = \sum_{x,y \in G}{\E_{s \in G^r}{\prod_{i=1}^r{1_{\mathcal{A}_i}(x,y)}}} = \alpha_1\cdots\alpha_r |G|^2.
\end{equation*}
By averaging we can pick some $s \in G^r$ such that $\mathcal{A}:=\bigcap_{i=1}^r{\mathcal{A}_i}$ has $|\mathcal{A}| \geq \alpha_1\cdots\alpha_r|G|^2$.

By the definition of $f_\Delta$ (from Theorem \ref{thm.triangle}) we have
\begin{equation*}
 \E_{x,y,z \in G}{1_{\mathcal{A}}(x,y)1_{\mathcal{A}}(zx,y)1_{\mathcal{A}}(x,yz)} = S(\mathcal{A}) \geq  f_\Delta(\alpha_1\cdots \alpha_r);
\end{equation*}
write
\begin{equation*}
Z:=\left\{z \in G: \E_{x,y \in G}{1_{\mathcal{A}}(x,y)1_{\mathcal{A}}(zx,y)1_{\mathcal{A}}(x,yz)} \geq \frac{1}{3}f(\alpha_1\cdots \alpha_r)\right\}.
\end{equation*}
Then
\begin{eqnarray*}
\P(Z) + \frac{1}{3}f_\Delta(\alpha_1\cdots \alpha_r) & \geq & \E_{x,y,z \in G}{1_{Z \sqcup (G \setminus Z)}(z)1_{\mathcal{A}}(x,y)1_{\mathcal{A}}(zx,y)1_{\mathcal{A}}(x,yz)}\\ & = & S(\mathcal{A}) \geq  f_\Delta(\alpha_1\cdots \alpha_r),
\end{eqnarray*}
and hence $\P(Z) \geq \frac{2}{3}f_\Delta(\alpha_1\cdots \alpha_r)$.  But then
\begin{eqnarray*}
\P(Z \setminus (A_{r+1} \cup \dots \cup A_{k})) & \geq & \frac{2}{3}f_\Delta(\alpha_1\cdots \alpha_r) - (\alpha_{r+1}+\dots+\alpha_k)\\ & \geq & \frac{1}{3}f_\Delta(\alpha_1\cdots \alpha_r)
\end{eqnarray*}
by (\ref{eqn.recurrence}).  Since $\bigcup_{i=1}^k{A_i} = G$ we conclude that there is some $1 \leq i \leq k$ such that
\begin{equation*}
\P((Z \setminus (A_{r+1} \cup \dots \cup A_{k}))\cap A_i) \geq \frac{1}{3r}f_\Delta(\alpha_1\cdots \alpha_r).
\end{equation*}
If course $(Z \setminus (A_{r+1} \cup \dots \cup A_{k}))\cap A_j=\emptyset$ for $r<j \leq k$ and so we may assume $i \leq r$.

Write $Z':= (Z \setminus (A_{r+1} \cup \dots \cup A_{k}))\cap A_i$. Since $Z' \subset Z$ we have
\begin{equation*}
\E_{x,y}{1_{\mathcal{A}_i}(x,y)1_{\mathcal{A}_i}(zx,y)1_{\mathcal{A}_i}(x,yz)}\geq \E_{x,y}{1_{\mathcal{A}}(x,y)1_{\mathcal{A}}(zx,y)1_{\mathcal{A}}(x,yz)} \geq \frac{1}{3}f_\Delta (\alpha_1\cdots \alpha_r)
\end{equation*}
for all $z \in Z'$.  On the other hand every $z \in Z'$ has $z \in A_i$ and so we conclude that there are at least
\begin{equation*}
 \frac{1}{3}f_\Delta (\alpha_1\cdots \alpha_r)|G|^2\cdot \frac{1}{3r}f_\Delta(\alpha_1\cdots \alpha_r)|G|
\end{equation*}
triples $(x,y,z) \in G^3$ such that
\begin{equation*}
z \in A_i, xs_iy \in A_i, zxs_iy \in A_i, \text{ and } xs_iyz \in A_i. 
\end{equation*}
The map $(x,y,z) \mapsto (xs_iy,z)$ has all fibres of size $|G|$ and so there are at least
\begin{equation*}
\frac{1}{9r}f_\Delta(\alpha_1\cdots \alpha_r)^2 |G|^2 \geq (g(\alpha_r))^2|G|^2
\end{equation*}
pairs $(a,b) \in G^2$ such that $a,b,ab,ba \in A_i$. This gives the result.
\end{proof}

\section{Proof of Proposition \ref{prop.col}}\label{sec.cy}

The key idea comes from Neumann's theorem \cite[Theorem 1]{neu::1} which is already identified in \cite[Footnote 4]{bertao::0}.  Neumann's theorem describes the structure of groups $G$ for which $c(G) \geq \epsilon$ -- they are the groups containing normal subgroups $K \leq H \leq G$ such that $K$ and $G/H$ have size $O_{\epsilon}(1)$ and $H/K$ is Abelian.  Neumann's theorem was further developed in \cite[Theorem 2.4]{ebe::1}, but both arguments provide a more detailed structure than we require.

We have made some effort to control the exponent; results such as \cite[Lemma 2.1]{ebe::1} or \cite[Theorem 2.2]{ols::0} could be used in place of Kemperman's Theorem in what follows at the possible expense of the $2$ becoming slightly larger.  Moving the $2+o_{\epsilon\rightarrow 0}(1)$ below $1$ would require a slightly different approach as we normalise a subgroup of index around $\epsilon^{-1}$ at a certain point which costs us a term of size $\epsilon^{-1}!$.
\begin{proposition*}[Proposition \ref{prop.col}]
Suppose that $G$ is a finite group and $c(G) \geq \epsilon$.  Then there is a cover $\mathcal{C}$ of $G$ of size $\exp((2+o_{\epsilon\rightarrow 0}(1))\epsilon^{-1}\log \epsilon^{-1})$ such that if $A \in \mathcal{C}$ and $(x,y,xy,yx) \in A^4$ then $xy=yx$.
\end{proposition*}
\begin{proof}
We work with the conjugation action of $G$ on itself (\emph{i.e.} $(g,x)\mapsto g^{-1}xg$), and write $x^G$ for the conjugacy class of $x$ (the orbit of $x$ under this action), and $C_G(x)$ for the centre of $x$ in $G$ (the stabiliser of $x$ under this action).

Let $\eta,\nu \in (0,1]$ be parameters (we shall take $\nu=\frac{1}{2}$ and $\eta = \epsilon/\log \epsilon^{-1})$) to be optimised later and put
\begin{equation*}
X:=\{x \in G: |x^G| \leq \eta^{-1}\}.
\end{equation*}
Then
\begin{equation*}
\epsilon |G|^2 \leq |G|^2\P(xy=yx) = \sum_{x}{|C_G(x)|} =|G|\sum_{x}{\frac{1}{|x^G|}}\leq \sum_{x \in X}{|G|} + \sum_{x \not \in X}{\eta|G|}.
\end{equation*}
Writing $\kappa:=|X|/|G|$ we can rearrange the above to see that $\kappa \geq (\epsilon-\eta)/(1-\eta)$.

Suppose that $s \in \N$ is maximal such that
\begin{equation*}
|\overbrace{X\cdots X}^{s \text{ times}}| \geq\left( 1+(1-\nu)(s-1)\right)|X|.
\end{equation*}
Certainly there is some $s \in \N$ since the inequality certainly holds for $s=1$, and there is a maximal such with $s \leq \frac{\kappa^{-1}-\nu}{1-\nu}$ since $|X| \geq \kappa |G|$.

Since $1_G^G=\{1_G\}$ we have $1_G \in X$ and $1_G \in X \cdots X$ for any $s$-fold product.  By Kemperman's Theorem \cite[Theorem 5]{kem::0} (also recorded on \cite[p111]{ols::1}, and which despite the additive notation does not assume commutativity) it follows that there is some $H \leq G$ such that
\begin{equation*}
|\overbrace{X\cdots X}^{s+1 \text{ times}}| \geq |\overbrace{X\cdots X}^{s \text{ times}}| + |X| - |H| \text{ and } H\subset \overbrace{X\cdots X}^{s+1 \text{ times}}.
\end{equation*}
By maximality of $s$ we know that
\begin{equation*}
\left( 1+(1-\nu)s\right)|X|>|\overbrace{X\cdots X}^{s+1 \text{ times}}|  \geq \left( 1+(1-\nu)(s-1)\right)|X|+|X|-|H|,
\end{equation*}
so that $|H| > \nu|X|$, and so $|G/H| < \nu^{-1}\kappa^{-1}$.

Let $K$ be the kernel of the action of left multiplication by $G$ on $G/H$ \emph{i.e.} $K:=\{x \in G: xgH=gH \text{ for all }g \in G\}$.  The action induces a homomorphism from $G$ to $\Sym(G/H)$ so that by the First Isomorphism Theorem
\begin{equation*}
K \triangleleft G \text{ and }|G/K| \leq |\Sym(G/H)| \leq |G/H|!.
\end{equation*}
Each $x \in H$ (and hence each $x \in K$ since $xH=H$ for such $x$) can be written as a product of $s+1$ elements of $X$.  Moreover, the function $x \mapsto |x^G|$ is sub-multiplicative \emph{i.e.} $|(xy)^G| \leq |x^G||y^G|$ and so it follows that
\begin{equation*}
|x^G| \leq \eta^{-(s+1)} \leq R:=\left\lfloor \eta^{-\frac{\kappa^{-1}+1-2\nu}{1-\nu}}\right\rfloor \text{ for all }x \in X^{s+1},
\end{equation*}
and in particular for all $x \in K$.  Thus for each $x \in K$ there is an injection $\phi_{x^G}:x^G \rightarrow \{1,\dots,R\}$.  With this notation we can define our covering: let 
\begin{equation*}
\mathcal{S}:=\{\{x\in K :\phi_{x^G}(x)=i \}: 1 \leq i \leq R\} \text{ and } \mathcal{C}:=((G/K) \setminus \{K\} )\cup \mathcal{S},
\end{equation*}
so that $\mathcal{S}$ is a cover of $K$ and $\mathcal{C}$ is a cover of $G$.  Now
\begin{align*}
|\mathcal{C}| \leq |G/K|-1 + |\mathcal{S}| & \leq \lfloor \nu^{-1}\kappa^{-1}\rfloor! -1 +R\\ & \leq\exp\left( \max\left\{\nu^{-1}\kappa^{-1}\log  \nu^{-1}\kappa^{-1},\frac{\kappa^{-1}+1-2\nu}{1-\nu}\log\eta^{-1}\right\}+O(1)\right).
\end{align*}
Optimise this by taking $\nu = \frac{1}{2}$ and $\eta=\epsilon/\log \epsilon^{-1}$ as mentioned before so that $\kappa \geq \epsilon (1-o_{\epsilon\rightarrow 0}(1))$ and $\log \eta^{-1} = (1+o_{\epsilon \rightarrow 0}(1))\log \epsilon^{-1}$.

Suppose that $A \in \mathcal{C}$ and $x,y,xy,yx \in A$.  If $A \in (G/K) \setminus \{K\}$ then $xK=yK=xyK=yxK=A$.  Since $K \triangleleft G$ we have $xK=xyK=(xK)(yK)$ and so $yK=K$ which is a contradiction.  It follows that $A \in \mathcal{S}$ and hence $x,y,xy,yx \in K$.  We conclude $\phi_{(xy)^G}(xy)=\phi_{(yx)^G}(yx)$ but $xy =y^{-1}(yx)y$ and so $(xy)^G = (yx)^G$.  Since $\phi_{(xy)^G}$ is an injection we conclude that $xy=yx$ as required.

The result is proved.
\end{proof}

\bibliographystyle{halpha}

\bibliography{references}

\end{document}